\documentclass[12pt]{amsart}

\usepackage{amscd,mathrsfs,epsfig,amsthm,amssymb,amsmath}
\usepackage[all]{xy} 
\usepackage[T1]{CJKutf8}
\usepackage[sans]{dsfont}  
\usepackage{hyperref}
\usepackage{tabularx}
\usepackage{rotating}

\usepackage{graphicx,xcolor} 
\usepackage{cite}
\usepackage{url}

\newtheorem{theorem}{Theorem}[subsection]
\newtheorem{corollary}[theorem]{Corollary}
\newtheorem{lemma}[theorem]{Lemma}
\newtheorem{definition}[theorem]{Definition}

\newtheorem{remark}[theorem]{Remark}

\newtheorem{example}[theorem]{Example}

\newtheorem*{thma}{Theorem A}
\newtheorem*{thmb}{Theorem B}

\allowdisplaybreaks[2]

\def\calC{{\mathcal C}}

\def\End{\mathop{\rm End}\nolimits} 
\def\Ext{\mathop{\rm Ext}\nolimits} 

\def\H{\mathop{\rm H}\nolimits}
\def\E{\mathop{\rm E}\nolimits}

\def\Hom{\mathop{\rm Hom}\nolimits}

\def\lim{\mathop{\varinjlim}\nolimits}
 
\def\Ob{\mathop{\rm Ob}\nolimits} 
\def\Mor{\mathop{\rm Mor}\nolimits}

\def\Res{\mathop{\rm Res}\nolimits}

\def\cod{\mathop{\rm cod}}
\def\dom{\mathop{\rm dom}}

\DeclareMathOperator{\op}{\rm op}

\begin{document}

\title{Extension category algebras and LHS--spectral sequences}

\author{Mawei Wu}
\address{School of Mathematics and Statistics, Lingnan Normal University, Zhanjiang, Guangdong 524048, China}
\email{wumawei@lingnan.edu.cn}



\thanks{The author \begin{CJK*}{UTF8}{}
\CJKtilde \CJKfamily{gbsn}(吴马威)
\end{CJK*} is supported by Lingnan Normal University (No. 000302503182)}


\begin{abstract}
Let $\mathcal{C}$ be a small category, $\mathfrak{A}$ be a precosheaf of unital $k$-algebras on $\mathcal{C}$ and $\mathfrak{M}$ be an $\mathfrak{A}$-bimodule. We introduce two new notions, namely, the Grothendieck construction $Gr_{\mathcal{C}}(\mathfrak{A}, \mathfrak{M})$ of $\mathfrak{A}$ and $\mathfrak{M}$, as well as the extension category algebra $\mathfrak{A} \ltimes \mathfrak{M}$. The extension category algebra contains the trivial extension algebra and the skew category algebra as special cases. If $\mathcal{C}$ is object-finite, we prove that the category of modules of $Gr_{\mathcal{C}}(\mathfrak{A}, \mathfrak{M})$ is equivalent to the category of modules over $\mathfrak{A} \ltimes \mathfrak{M}$. Finally, we obtain two LHS-spectral sequences about $Gr_{\mathcal{C}}(\mathfrak{A}, \mathfrak{N})$ for a right $\mathfrak{A}$-module $\mathfrak{N}$.   
\end{abstract}

\maketitle

\tableofcontents

\section{Introduction}
The trivial extension algebra (see Definition \ref{triext}) has been extensively studied and used in various areas of research, such as representation theory, cohomology theory, category theory and homological algebra \cite{MY20,FGR06,Lof76,PR73,AI86,MP00,AHR84}, just to name a few. There are some generalizations of the trivial extension algebras, see \cite{Pog05,ABFS17,BBG20}.

The skew category algebra (see Defintion \ref{skewcatalg}) is an important class of algebras. It was used in \cite{WX23} to characterize the category of modules on ringed sites. A generalization of the skew category algebras was given in \cite[Definition 2.2.2]{Wu24}, so-called the pseudoskew category algebras. There are some other works about skew category algebras, see \cite{CMT24,Bav17,Bav20}.

In this paper, we introduce a new class of algebras, so-called the extension category algebras. The trivial extension algebras and the skew category algebras are some special examples of the extension category algebras. We give a characterization of the category of modules over the extension category algebras, and establish two LHS-spectral sequences.

To be more precise, let $\calC$ be a small category, $\mathfrak{A} \colon \calC \to k\mbox{\rm -Alg}$ be a precosheaf of unital $k$-algebras on $\calC$, and let $\mathfrak{M} \colon \calC \to k\mbox{\rm -Mod}$ be an $\mathfrak{A}$-bimodule (see Defintion \ref{biamod}), we first define the Grothendieck construction $Gr_{\mathcal{C}}(\mathfrak{A},\mathfrak{M})$ (see Definition \ref{bigrocon}) of $\mathfrak{A}$ and $\mathfrak{M}$, as well as the extension category algebra $\mathfrak{A} \ltimes \mathfrak{M}$ (see Definition \ref{extcatalg}) with respect to $\mathfrak{A}$ and $\mathfrak{M}$. Then we show that the module categories of $Gr_{\mathcal{C}}(\mathfrak{A},\mathfrak{M})$ and the category of modules over $\mathfrak{A} \ltimes \mathfrak{M}$ are equivalent.

\begin{thma} (Theorem \ref{main})
Let $\calC$ be a small category, $\mathfrak{A} \colon \calC \to k\mbox{\rm -Alg}$ be a precosheaf of unital $k$-algebras on $\calC$, and let $\mathfrak{M} \colon \calC \to k\mbox{\rm -Mod}$ be an $\mathfrak{A}$-bimodule. If $\Ob \mathcal{C}$ is finite, then there is a category equivalence
$$
Gr_{\mathcal{C}}(\mathfrak{A}, \mathfrak{M}) \mbox{\rm -Mod} \simeq \mathfrak{A} \ltimes \mathfrak{M} \mbox{\rm -Mod}.
$$     
\end{thma}

Given a functor $\mathfrak{F} \colon \mathcal{C} \to \mbox{\rm Cat}$, the authors in \cite{PJ06} construct a spectral sequence abutting to the Baues-Wirsching cohomology of the Grothendieck construction of $\mathfrak{F}$, in
terms of the cohomology of $\mathcal{C}$ and of $\mathfrak{F}(x)$ for $x \in \Ob \mathcal{C}$. The authors in \cite{KY25} also construct a similar spectral sequence that converges to the Thomason cohomology of the Grothendieck construction. In this paper, we try to establish some similar spectral sequences for the Grothendieck construction $Gr_{\mathcal{C}}(\mathfrak{A},\mathfrak{M})$. Let $\mathfrak{N}$ be a right $\mathfrak{A}$-module (see Definition \ref{ramod}), by using some results of Xu \cite{Xu08} and Yal{\c{c}}{\i}n \cite{Yal24}, we establish two LHS-spectral sequences about the Grothendieck construction $Gr_{\mathcal{C}}(\mathfrak{A},\mathfrak{N})$ (see Definition \ref{rgrocon}) of $\mathfrak{A}$ and $\mathfrak{N}$. 

In the following theorem, $\mathfrak{N}_{\sqcup}:=\bigsqcup_{x \in \Ob \calC} \mathfrak{N}(x)$ is a disjoint union of the underlying abelian groups (viewed as categories \cite{Mit72}) of $k$-modules $\mathfrak{N}(x)$. The category $Gr_{\mathcal{C}}(\mathfrak{A})$ is the Grothendieck construction of $\mathfrak{A}$ (see Definition \ref{algrocon}) and $\mathbb{H}^*(\mathfrak{N}_{\sqcup}; \mathfrak{F})$ is a right $Gr_{\mathcal{C}}(\mathfrak{A})$-module (see Definition \ref{hfun}). The functor $\pi$ is defined as in Definition \ref{exofcat} and $\Res_{\pi}$ is the restriction functor along $\pi$. For the definitions of Ext-groups and cohomology group of $\mathcal{C}$, one can see Definition \ref{extdef} and Definition \ref{cohdef} respectively.

\begin{thmb} (Theorem \ref{lhs1} and Theorem \ref{lhs2})
Let $\mathfrak{F} \in \mbox{\rm Mod-} Gr_{\mathcal{C}}(\mathfrak{A}, \mathfrak{N})$ and $\mathfrak{G} \in \mbox{\rm Mod-} Gr_{\mathcal{C}}(\mathfrak{A})$, then there are two spectral sequences
    $$
    \E_2^{p,q}=\Ext^p_{Gr_{\mathcal{C}}(\mathfrak{A})}(\mathfrak{G}, \mathbb{H}^q(\mathfrak{N}_{\sqcup}; \mathfrak{F})) \Longrightarrow \Ext^{p+q}_{Gr_{\mathcal{C}}(\mathfrak{A}, \mathfrak{N})}(\Res_{\pi} \mathfrak{G}, \mathfrak{F}),
    $$
and
    $$
    \E_2^{p,q}=\H^p(Gr_{\mathcal{C}}(\mathfrak{A}); \mathbb{H}^q(\mathfrak{N}_{\sqcup}; \mathfrak{F})) \Longrightarrow \H^{p+q}(Gr_{\mathcal{C}}(\mathfrak{A}, \mathfrak{N}); \mathfrak{F}).
    $$    
\end{thmb}

This paper is organized as follows. In Section \ref{prelim}, some basic definitions are recalled, including the trivial extension algebras and the skew category algebras, the modules over precosheaves of algebras, the representations and cohomology of categories, and the extensions of categories. In Section \ref{cha}, two new notions are introduced, that is, the Grothendieck construction $Gr_{\mathcal{C}}(\mathfrak{A,\mathfrak{M}})$ and the extension category algebra $\mathfrak{A} \ltimes \mathfrak{M}$. Then we prove that the category of modules of $Gr_{\mathcal{C}}(\mathfrak{A,\mathfrak{M}})$ and the category of modules over $\mathfrak{A} \ltimes \mathfrak{M}$ are equivalent. Finally, we use this main characterization theorem to reprove \cite[Theorem A]{WX23}. In Section \ref{lhs}, two LHS-spectral sequences of $Gr_{\mathcal{C}}(\mathfrak{A},\mathfrak{N})$ are established for a right $\mathfrak{A}$-module $\mathfrak{N}$.

\noindent\textbf{Conventions and Notations:}
\begin{enumerate}
    \item Throughout $k$ is a unital commutative ring.
    \item All algebras and their morphisms are unital, $k\mbox{\rm -Alg}$ is the category of unital algebras and unital algebra homomorphisms.
    \item $\mathfrak{A} [\mathcal{C}]$ (resp. $\mathfrak{A} \ltimes \mathfrak{M}$) is a skew category algebra (resp. extension category algebra).
    \item $\Ob \mathcal{C}$ (resp.  $\Mor \mathcal{C}$) is the set of all objects (resp. morphisms) of $\mathcal{C}$.
    \item For a morphism $f \colon x \to y$, its domain (resp. codomain) is denoted by $\dom(f)$ (resp. $\cod(f)$). 
    \item The composition of maps $x \overset{f}{\longrightarrow} y \overset{g}{\longrightarrow} z$ is written as $fg$. (from left to right!)
    \item $1_{\mathfrak{A}(x)}$ is the identity element of the $k$-algebra $\mathfrak{A}(x)$, $0_{\mathfrak{M}(x)}$ is the zero element of the $k$-module $\mathfrak{M}(x)$, $1_x$ is the identity morphism of $x \in \Ob \mathcal{C}$. 
    \item $k\mathcal{C}$ is the $k$-linearization of the category $\mathcal{C}$.
    \item $\mbox{\rm Fun}(\mathcal{C}, \mathcal{D})$ is the category of functors from $\mathcal{C}$ to $\mathcal{D}$.
    \item $\mbox{\rm Mod-} \mathcal{C}$ (resp. $\mathcal{C} \mbox{\rm -Mod}$) is the category of right (resp. left) $\mathcal{C}$-modules
    \item $\underline{C}$ is the constant functor with value at $C$.
    \item $Gr_{\mathcal{C}}(\mathfrak{A})$ is the Grothendieck construction of the functor $\mathfrak{A}$.
    \item The action is denoted by $\cdot$.
\end{enumerate}

\section{Preliminaries} \label{prelim}
In this section, some basic notions are recalled. To be more precise, we recall the following definitions: trivial extension algebras and skew category algebras, the (bi/right) modules over precosheaves of algebras, the representations and cohomology of categories, and the extensions of categories. 

\subsection{Trivial extension algebras and skew category algebras}
For the convenience of the reader, the definitions of the trivial extension algebras and the skew category algebras are recorded.

\begin{definition} \cite[p.78]{ARS95} \label{triext}
Let $\Lambda$ be a unital $k$-algebra and $M$ be a $\Lambda$-bimodule. The
$k$-linear space $\Lambda \oplus M$ equipped with the multiplication given by
$$
(a_1,m_1) (a_2,m_2) = (a_1a_2, a_1m_2+m_1a_2)
$$
for all $a_1, a_2 \in \Lambda$ and $m_1, m_2 \in M$, form a $k$-algebra. Such a $k$-algebra is called a \emph{trivial extension algebra of $\Lambda$ by
$M$} and it is denoted by $\Lambda \ltimes M$.
\end{definition}

The co-variant version of the skew category algebras is given as follows.

\begin{definition} (co-variant version of \cite[Definition 3.2.1]{WX23}) \label{skewcatalg}
	Let $\calC$ be a (non-empty) small category. Let $\mathfrak{A} \colon \calC \to k\text{-}{\rm Alg}$ be a precosheaf of $k$-algebras. The \emph{skew category algebra} $\mathfrak{A}[\calC]$ on $\calC$ with respect to $\mathfrak{A}$ is a $k$-module spanned over elements of the form $rf$, where $f \in \Mor \calC$ and $r \in \mathfrak{A}(\cod(f))$. We define the multiplication on two base elements by the rule
	\begin{eqnarray}
		  sg \ast rf =
		\begin{cases}
			(\mathfrak{A}(g)(r)s) fg,       & \text{if}\ \dom(g)=\cod(f); \notag \\
			0, & {\rm otherwise}.
		\end{cases}
	\end{eqnarray} 
Extending this product linearly to two arbitrary elements, $\mathfrak{A}[\calC]$ becomes an associative $k$-algebra.
\end{definition}

\subsection{Modules over precosheaves of algebras}
Let $\mathfrak{A} \colon \calC \to k\mbox{\rm -Alg}$ be a precosheaf of unital $k$-algebras, the definitions of $\mathfrak{A}$-bimodules and right $\mathfrak{A}$-modules are recalled in this subsection.

\begin{definition} \label{biamod}
Let $\calC$ be a small category, $\mathfrak{A} \colon \calC \to k\mbox{\rm -Alg}$ be a precosheaf of unital $k$-algebras on $\calC$, then $\mathfrak{M} \colon \calC \to k\mbox{\rm -Mod}$ is called an \emph{$\mathfrak{A}$-bimodule} if it also satisfies the following conditions:
\begin{enumerate}
    \item for each $x \in \Ob \mathcal{C}$, $\mathfrak{M}(x)$ is an $\mathcal{A}(x)$-bimodule,
    \item for each $f \colon x \to y$ and $r, s \in \mathfrak{A}(x)$, $m \in \mathfrak{M}(x)$, 
    $$
    \mathfrak{M}(f)(r \cdot_l m)=\mathfrak{A}(f)(r) \cdot_l \mathfrak{M}(f)(m),
    $$
    and
    $$
    \mathfrak{M}(f)(m \cdot_r s)= \mathfrak{M}(f)(m) \cdot_r \mathfrak{A}(f)(s),
    $$
    where $\cdot_l$ (resp. $\cdot_r$) is the left (resp. right) $\mathfrak{A}(x)$-action on $\mathfrak{M}(x)$. For simplicity, the subscripts will be omitted later.
\end{enumerate}
\end{definition}

Similarly, forgetting the left $\mathfrak{A}$-action, we have the definition of right $\mathfrak{A}$-modules. 

\begin{definition} \label{ramod}
Let $\calC$ be a small category, $\mathfrak{A} \colon \calC \to k\mbox{\rm -Alg}$ be a precosheaf of unital $k$-algebras on $\calC$, then $\mathfrak{N} \colon \calC \to k\mbox{\rm -Mod}$ is called a \emph{right $\mathfrak{A}$-module} if it also satisfies the following conditions:
\begin{enumerate}
    \item for each $x \in \Ob \mathcal{C}$, $\mathfrak{N}(x)$ is a right $\mathcal{A}(x)$-module,
    \item for each $f \colon x \to y$ and $s \in \mathfrak{A}(x)$, $m \in \mathfrak{N}(x)$, 
    $$
    \mathfrak{N}(f)(m \cdot_r s)= \mathfrak{N}(f)(m) \cdot_r \mathfrak{A}(f)(s).
    $$
\end{enumerate}   
\end{definition}

\subsection{Representations and cohomology of categories}
The definitions of the modules and cohomology of a category $\mathcal{C}$, as well as the Ext-groups of two modules of $\mathcal{C}$ are given in this subsection. For more information, see \cite{Web07,Yal24,Xu08}. 

\begin{definition}
Let $k$ be a commutative ring with unity. A contra-variant functor $\mathfrak{F} \colon \mathcal{C}^{\op} \to k\mbox{\rm -Mod}$ is called a \emph{right $\mathcal{C}$-module}. A co-variant functor $\mathfrak{F} \colon \mathcal{C} \to k\mbox{\rm -Mod}$ is called a \emph{left $\mathcal{C}$-module}.    
\end{definition}

All the right (resp. left) $\mathcal{C}$-modules and the natural morphisms between them form a functor category, we denote it by $\mbox{\rm Mod-} \mathcal{C} := \mbox{\rm Fun}(\mathcal{C}^{\op}, k\mbox{-Mod})$ (resp. $\mathcal{C} \mbox{\rm -Mod} := \mbox{\rm Fun}(\mathcal{C}, k\mbox{-Mod})$).

For two $\mathcal{C}$-modules $\mathfrak{F}$ and $\mathfrak{G}$, their Ext-group is defined as follows.

\begin{definition} \label{extdef}
Let $\mathfrak{F}, \mathfrak{G}$ be two $\mathcal{C}$-modules. For every $n \geq 0$, \emph{the $\Ext$-group of $\mathfrak{G}$ and
$\mathfrak{F}$} is defined by
$$
\Ext^n_{\mathcal{C}}(\mathfrak{G},\mathfrak{F}):=[R^n\Hom_{\mathcal{C}}(\mathfrak{G},-)](\mathfrak{F}),
$$
where $R^n\Hom_{\mathcal{C}}(\mathfrak{G},-)$ is the $n$-th right derived functor of $\Hom_{\mathcal{C}}(\mathfrak{G},-)$.
\end{definition}

Let $\underline{k}$ be a constant functor, then the cohomolgy group of $\mathcal{C}$ is defined as follows.

\begin{definition} \label{cohdef}
For every $\mathcal{C}$-module $\mathfrak{F}$, for $n \geq 0$, \emph{the $n$-th cohomology group of $\mathcal{C}$ with coefficients in $\mathfrak{F}$} is defined by
$$
\H^n(\mathcal{C}; \mathfrak{F}) := \Ext^n_{\mathcal{C}}(\underline{k}, \mathfrak{F}).
$$
\end{definition}

\subsection{Extensions of categories}
There are some different versions of the extensions of categories, see \cite{Web07,Hof94,BW84}. The extensions of categories considered in the present paper are in the sense of Hoff.

\begin{definition} \cite[D{\'e}finition 1.1]{Hof94} \label{extension}
An \emph{extension $\mathcal{E}$ of a category $\mathcal{C}$ via a category $\mathcal{K}$} is a sequence of functors
$$
\mathcal{K} \overset{\iota}{\longrightarrow} \mathcal{E} \overset{\pi}{\longrightarrow} \mathcal{C},
$$
which has the following properties: 
\begin{enumerate}
    \item $\Ob \mathcal{K} = \Ob \mathcal{E} = \Ob \mathcal{C}$, the functor $\iota$ is injective and $\pi$ is surjective on morphisms, both $\iota$ and $\pi$ are identities on objects;
    \item for two morphisms $f, g \in \Mor \mathcal{E}$, $\pi(f) = \pi(g)$ if and only if there is a unique $h \in \Mor \mathcal{K}$ such that $f \circ \iota(h) = g$.
\end{enumerate}
\end{definition}

\section{Extension category algebras and their modules} \label{cha}
Let $\mathfrak{A} \colon \calC \to k\mbox{\rm -Alg}$ be a precosheaf of unital $k$-algebras and $\mathfrak{M} \colon \calC \to k\mbox{\rm -Mod}$ be an $\mathfrak{A}$-bimodule. In this section, we first introduce two new notions: the Grothendieck construction $Gr_{\mathcal{C}}(\mathfrak{A,\mathfrak{M}})$ of $\mathfrak{A}$ and $\mathfrak{M}$, as well as the extension category algebra $\mathfrak{A} \ltimes \mathfrak{M}$ with respect to $\mathfrak{A}$ and $\mathfrak{M}$. Then we prove that the category of modules of $Gr_{\mathcal{C}}(\mathfrak{A,\mathfrak{M}})$ and the category of modules over $\mathfrak{A} \ltimes \mathfrak{M}$ are equivalent. Finally, we use this characterization theorem to reprove \cite[Theorem A]{WX23}.

\subsection{Grothendieck constructions}
Given a precosheaf of unital $k$-algebras $\mathfrak{A}$ and  an $\mathfrak{A}$-bimodule $\mathfrak{M}$, we define a category as follows. Roughly speaking, it put all the information of $\mathfrak{A}$ and $\mathfrak{M}$ together.

\begin{definition} \label{bigrocon}
Let $\calC$ be a small category, $\mathfrak{A} \colon \calC \to k\mbox{\rm -Alg}$ be a precosheaf of unital $k$-algebras on $\calC$, and $\mathfrak{M} \colon \calC \to k\mbox{\rm -Mod}$ be an $\mathfrak{A}$-bimodule. Then a category $Gr_{\mathcal{C}}(\mathfrak{A}, \mathfrak{M})$, called \emph{the Grothendieck construction of $\mathfrak{A}$ and $\mathfrak{M}$}, is defined as follows:
\begin{enumerate}
    \item $\Ob Gr_{\mathcal{C}}(\mathfrak{A}, \mathfrak{M}):= \Ob \calC$;
    \item for each $x, y \in \Ob Gr_{\mathcal{C}}(\mathfrak{A}, \mathfrak{M})$,
    $$
    \Hom_{Gr_{\mathcal{C}}(\mathfrak{A}, \mathfrak{M})}(x, y) := \left\{ (r,m,f) ~|~ f \in \Hom_{\mathcal{C}}(x,y), r \in \mathfrak{A}(y), m \in \mathfrak{M}(y) \right\},
    $$
    \item for each $x, y, z \in \Ob Gr_{\mathcal{C}}(\mathfrak{A}, \mathfrak{M})$, and each $(r,m,f) \in \Hom_{Gr_{\mathcal{C}}(\mathfrak{A}, \mathfrak{M})}(x, y)$, $(s,n,g) \in \Hom_{Gr_{\mathcal{C}}(\mathfrak{A}, \mathfrak{M})}(y, z)$, we set
    $$
    (r,m,f) \circ (s,n,g) : = \left( \mathfrak{A}(g)(r)s, \mathfrak{A}(g)(r) \cdot n + \mathfrak{M}(g)(m) \cdot s, fg \right).
    $$
\end{enumerate}
\end{definition}

\begin{remark}
\begin{enumerate}
    \item It seems weird that one takes a Grothendieck construction for ``two'' functors. In some sense, the construction in the definition above seems like combining the Grothendieck construction of $\mathfrak{A}$ with the  Grothendieck construction of $\mathfrak{M}$ suitably, hence the name ``Grothendieck construction of $\mathfrak{A}$ and $\mathfrak{M}$'' is taken.
    \item The category $Gr_{\mathcal{C}}(\mathfrak{A}, \mathfrak{M})$ is non-additive, since the Hom set $\Hom_{Gr_{\mathcal{C}}(\mathfrak{A}, \mathfrak{M})}(x, y)$ has no natural abelian group structure.
\end{enumerate}
\end{remark}

\subsection{Extension category algebras}
Motivated by the definitions of the trivial extension algebras and the skew category algebras, we introduce a new class of algebras.

\begin{definition} \label{extcatalg}
Let $\calC$ be a (non-empty) small category. Let $\mathfrak{A} \colon \calC \to k\text{-}{\rm Alg}$ be a precosheaf of $k$-algebras and $\mathfrak{M} \colon \calC \to k\mbox{\rm -Mod}$ be an $\mathfrak{A}$-bimodule. The \emph{extension category algebra} $\mathfrak{A} \ltimes \mathfrak{M}$ on $\calC$ with respect to $\mathfrak{A}$ and $\mathfrak{M}$ is a $k$-module spanned over elements of the form $\left\{~rmf ~|~ f \in \Hom_{\calC}(x, y), r \in \mathfrak{A}(y), m \in \mathfrak{M}(y)~ \right\}$. We define the multiplication on two base elements by the rule
	\begin{eqnarray}
		sng \ast rmf =
		\begin{cases}
			t w h ,       & \text{if}\ \dom(g)=\cod(f); \notag \\
			0, & {\rm otherwise},
		\end{cases}
	\end{eqnarray}
where $t=\mathfrak{A}(g)(r)s$, $w=\mathfrak{A}(g)(r) \cdot n + \mathfrak{M}(g)(m) \cdot s$, $h=fg$. 
Extending this product linearly to two arbitrary elements, $\mathfrak{A} \ltimes \mathfrak{M}$ becomes an associative $k$-algebra.
\end{definition}

\begin{example} \label{ex}
\begin{enumerate}
    \item If $\mathcal{C}$ is a trivial category with single object $\bullet$, then $\mathfrak{A} \ltimes \mathfrak{M}$ is a trivial extension algebra $\mathfrak{A(\bullet)} \ltimes \mathfrak{M}(\bullet)$ (see Definition \ref{triext}).
    \item If $\mathfrak{M}=\underline{0}$, then $\mathfrak{A} \ltimes \mathfrak{M}$ is a skew category algebra $\mathfrak{A}[\mathcal{C}]$ (see Definition \ref{skewcatalg}).
\end{enumerate}    
\end{example}

\begin{remark}
\begin{enumerate}
    \item The algebra $\mathfrak{A} \ltimes \mathfrak{M}$ has an identity $\sum_{x \in \Ob \mathcal{C}} 1_{\mathfrak{A}(x)}0_{\mathfrak{M}(x)}1_x$ if $\mathcal{C}$ is object-finite (i.e. $\Ob \mathcal{C} < + \infty$).
    \item Due to the Example \ref{ex} (1) above, the extension category algebras $\mathfrak{A} \ltimes \mathfrak{M}$ can be viewed as ``trivial extension algebras with several objects''.
\end{enumerate}     
\end{remark}

\subsection{A characterization of modules over extension category algebras}
In this subsection, a characterization of the category of modules over an extension category algebra is given. Before that, we first recall the following well-known lemma.

\begin{lemma} \label{linear}
Let $k\mathcal{C}$ be a $k$-linearization category of $\mathcal{C}$, $\mbox{\rm Fun}_k(k\mathcal{C}, k\mbox{\rm -Mod})$ be the category of $k$-linear functors from $k\mathcal{C}$ to $k\mbox{\rm -Mod}$, then we have the following category equivalence
$$
\mbox{\rm Fun}(\mathcal{C}, k\mbox{\rm -Mod}) \simeq \mbox{\rm Fun}_k(k\mathcal{C}, k\mbox{\rm -Mod}).
$$    
\end{lemma}

Now, it is time to give our first main result in this paper.

\begin{theorem} \label{main}
Let $\calC$ be a small category, $\mathfrak{A} \colon \calC \to k\mbox{\rm -Alg}$ be a precosheaf of unital $k$-algebras on $\calC$, and let $\mathfrak{M} \colon \calC \to k\mbox{\rm -Mod}$ be an $\mathfrak{A}$-bimodule. If $\Ob \mathcal{C}$ is finite, then there is a category equivalence
$$
Gr_{\mathcal{C}}(\mathfrak{A}, \mathfrak{M}) \mbox{\rm -Mod} \simeq \mathfrak{A} \ltimes \mathfrak{M} \mbox{\rm -Mod}.
$$    
\end{theorem}

\begin{proof}
By Lemma \ref{linear}, we have
$$
Gr_{\mathcal{C}}(\mathfrak{A}, \mathfrak{M}) \mbox{\rm -Mod} = \mbox{\rm Fun}(Gr_{\mathcal{C}}(\mathfrak{A}, \mathfrak{M}), k\mbox{\rm -Mod}) \simeq \mbox{\rm Fun}_k(kGr_{\mathcal{C}}(\mathfrak{A}, \mathfrak{M}), k\mbox{\rm -Mod}),
$$  
where $kGr_{\mathcal{C}}(\mathfrak{A}, \mathfrak{M})$ is the $k$-linearization of $Gr_{\mathcal{C}}(\mathfrak{A}, \mathfrak{M})$, hence it is an additive category.
Let $\mathfrak{G}$ be the following functor
$$
\mathfrak{G}:=\bigoplus_{x \in \Ob kGr_{\mathcal{C}}(\mathfrak{A}, \mathfrak{M})}\Hom_{kGr_{\mathcal{C}}(\mathfrak{A}, \mathfrak{M})}(x, -).
$$
Since $\Ob \mathcal{C}$ is finite, it follows that $\Ob kGr_{\mathcal{C}}(\mathfrak{A}, \mathfrak{M})$ is finite too. 

Now, let us compute $\End(\mathfrak{G})$:

\begin{align*}
     & \End(\mathfrak{G}) \\
   = & \Hom\left(\bigoplus_{y \in \Ob kGr_{\mathcal{C}}(\mathfrak{A}, \mathfrak{M})}\Hom_{kGr_{\mathcal{C}}(\mathfrak{A}, \mathfrak{M})}(y, -), \bigoplus_{x \in \Ob kGr_{\mathcal{C}}(\mathfrak{A}, \mathfrak{M})}\Hom_{kGr_{\mathcal{C}}(\mathfrak{A}, \mathfrak{M})}(x, -)\right)\\
   \cong & \bigoplus_{x \in \Ob kGr_{\mathcal{C}}(\mathfrak{A}, \mathfrak{M})} \Hom\left(\bigoplus_{y \in \Ob kGr_{\mathcal{C}}(\mathfrak{A}, \mathfrak{M})}\Hom_{kGr_{\mathcal{C}}(\mathfrak{A}, \mathfrak{M})}(y, -), \Hom_{kGr_{\mathcal{C}}(\mathfrak{A}, \mathfrak{M})}(x, -)\right)  \\
    \cong & \bigoplus_{x \in \Ob kGr_{\mathcal{C}}(\mathfrak{A}, \mathfrak{M})} \bigoplus_{y \in \Ob kGr_{\mathcal{C}}(\mathfrak{A}, \mathfrak{M})} \Hom\left(\Hom_{kGr_{\mathcal{C}}(\mathfrak{A}, \mathfrak{M})}(y, -), \Hom_{kGr_{\mathcal{C}}(\mathfrak{A}, \mathfrak{M})}(x, -)\right) \\
     \cong & \bigoplus_{x \in \Ob kGr_{\mathcal{C}}(\mathfrak{A}, \mathfrak{M})} \bigoplus_{y \in \Ob kGr_{\mathcal{C}}(\mathfrak{A}, \mathfrak{M})} \Hom_{kGr_{\mathcal{C}}(\mathfrak{A}, \mathfrak{M})}(x, y) \\
     \cong & (\mathfrak{A} \ltimes \mathfrak{M}) ^{\op}.  
\end{align*}

The first and second isomorphisms hold since $\Ob kGr_{\mathcal{C}}(\mathfrak{A}, \mathfrak{M})$ is finite, and $\Hom$ functor preserves finite (co)products. The third isomorphism holds due to Yoneda Lemma. The last isomorphism is proved by considering the map $\Phi\colon  (r,m,f) \mapsto rmf.$ It is not hard to see that this map is a bijection. The map $\Phi$ is an isomorphism because it also preserves the multiplication: let
$t=\mathfrak{A}(g)(r)s$, $w=\mathfrak{A}(g)(r) \cdot n + \mathfrak{M}(g)(m) \cdot s$ and $h=fg$, then
\begin{align*}
     & \Phi((r,m,f) \circ (s,n,g)) & \\ 
    = & \Phi \left( (t,w,h)) \right) & (Definition~\ref{bigrocon}~(3)) \\
    = & twh & (Definition~of~\Phi) \\
    = & sng \ast rmf & (Definition~\ref{extcatalg}) \\
    = & \Phi((s,n,g)) \ast \Phi((r,m,f)) & (Definition~of~\Phi)  \\
    = & \Phi((r,m,f)) \ast^{\op} \Phi((s,n,g)). & \\
\end{align*}

Thus, by \cite[Exercise F on p.106]{Fre64}, there is a category equivalence
$$
Gr_{\mathcal{C}}(\mathfrak{A}, \mathfrak{M}) \mbox{\rm -Mod} \simeq \mathfrak{A} \ltimes \mathfrak{M} \mbox{\rm -Mod}.
$$ 
This completes the proof.
\end{proof}

With the theorem above, we reprove \cite[Theorem A]{WX23}. In some sense, the theorem above can be seen as a generalization of \cite[Theorem A]{WX23}. 

\begin{corollary} (see \cite[Theorem A]{WX23})
Let $\mathcal{C}$ be a small category and $\mathfrak{A} \colon \calC \to k\text{-}{\rm Alg}$ be a precosheaf of $k$-algebras on $\mathcal{C}$. If $\Ob \mathcal{C}$ is finite, then we have the following category equivalence  
$$
\mathfrak{A} \mbox{\rm -} \mathfrak{M}\mbox{\rm od} \simeq \mathfrak{A}[\mathcal{C}] \mbox{\rm -Mod},
$$
where $\mathfrak{A}[\mathcal{C}]$ is the skew category algebra with respect to $\mathfrak{A}$.
\end{corollary}

\begin{proof}
Let $\mathfrak{M}=\underline{0}$, then $\mathfrak{A} \ltimes \mathfrak{M} = \mathfrak{A}[\mathcal{C}]$ (see Example \ref{ex} (2)) and $Gr_{\mathcal{C}}(\mathfrak{A}, \mathfrak{M}) = Gr_{\mathcal{C}}(\mathfrak{A})$. By Theorem \ref{main}, we have $Gr_{\mathcal{C}}(\mathfrak{A}) \mbox{\rm -Mod} \simeq \mathfrak{A}[\mathcal{C}] \mbox{\rm -Mod}$. Then, by Lemma \ref{linear}, we have $kGr_{\mathcal{C}}(\mathfrak{A}) \mbox{\rm -Mod} \simeq Gr_{\mathcal{C}}(\mathfrak{A}) \mbox{\rm -Mod} \simeq \mathfrak{A}[\mathcal{C}] \mbox{\rm -Mod}$. Furthermore, Howe (see \cite[Proposition 5]{How81}) showed that $\mathfrak{A} \mbox{\rm -} \mathfrak{M}\mbox{\rm od} \simeq kGr_{\mathcal{C}}(\mathfrak{A}) \mbox{\rm -Mod}$ (note that the semi-direct product construction therein is in fact the linear Grothendieck construction of $\mathfrak{A}$, we should also replace the abelian group-valued functor in \cite[Proposition 5]{How81} with the $k$-module-valued functor for applying to our situation). Therefore, we have $\mathfrak{A} \mbox{\rm -} \mathfrak{M}\mbox{\rm od} \simeq \mathfrak{A}[\mathcal{C}] \mbox{\rm -Mod}$. This completes the proof. 
\end{proof}

\section{Generalized Lyndon–Hochschild–Serre spectral sequences} \label{lhs}
For a right $\mathfrak{A}$-module $\mathfrak{N}$, two LHS-spectral sequences about $Gr_{\mathcal{C}}(\mathfrak{A},\mathfrak{N})$ will be established in this section.

\subsection{An extension of categories}
In this subsection, an extension of categories about $Gr_{\mathcal{C}}(\mathfrak{A},\mathfrak{N})$ will be built. Let us first recall the definition of Grothendieck construction of a precosheaf $\mathfrak{A} \colon \calC \to k\mbox{\rm -Alg}$ of unital $k$-algebras. 

\begin{definition} \cite[p.205]{WX23} \label{algrocon}
Let $\calC$ be a small category, $\mathfrak{A} \colon \calC \to k\mbox{\rm -Alg}$ be a precosheaf of unital $k$-algebras on $\calC$. Then a category $Gr_{\mathcal{C}}(\mathfrak{A})$, called \emph{the Grothendieck construction of $\mathfrak{A}$}, is defined as follows:
\begin{enumerate}
    \item $\Ob Gr_{\mathcal{C}}(\mathfrak{A}):= \Ob \calC$;
    \item for each $x, y \in \Ob Gr_{\mathcal{C}}(\mathfrak{A})$,
    $$
    \Hom_{Gr_{\mathcal{C}}(\mathfrak{A})}(x, y) := \left\{ (r,f) ~|~ f \in \Hom_{\mathcal{C}}(x,y), r \in \mathfrak{A}(y) \right\}
    $$
    \item for each $x, y, z \in \Ob Gr_{\mathcal{C}}(\mathfrak{A})$, and each $(r,f) \in \Hom_{Gr_{\mathcal{C}}(\mathfrak{A})}(x, y)$, $(s,g) \in \Hom_{Gr_{\mathcal{C}}(\mathfrak{A})}(y, z)$, we set
    $$
    (r,f) \circ (s,g) : = \left( \mathfrak{A}(g)(r)s, fg \right).
    $$
\end{enumerate}
\end{definition}

Let $\mathfrak{M}_{\sqcup}:=\bigsqcup_{x \in \Ob \calC} \mathfrak{M}(x)$ be a disjoint union of the underlying abelian groups of $k$-modules $\mathfrak{M}(x)$, viewing it as a category, then there are two natural functors 
$$
\mathfrak{M}_{\sqcup} \overset{\iota}{\longrightarrow} Gr_{\mathcal{C}}(\mathfrak{A}, \mathfrak{M}) \overset{\pi}{\longrightarrow} Gr_{\mathcal{C}}(\mathfrak{A}),
$$ 
where the functor 
$
\iota \colon \mathfrak{M}_{\sqcup} \longrightarrow Gr_{\mathcal{C}}(\mathfrak{A}, \mathfrak{M})
$
is defined as:
$
x \overset{m}{\longrightarrow} x \mapsto  x \overset{(1_{\mathfrak{A}(x)}, m, 1_x)}{\longrightarrow} x,
$
and the functor
$
\pi \colon Gr_{\mathcal{C}}(\mathfrak{A}, \mathfrak{M}) \longrightarrow Gr_{\mathcal{C}}(\mathfrak{A})
$
is defined as:
$
x \overset{(r,m,f)}{\longrightarrow} y \mapsto x \overset{(r,f)}{\longrightarrow} y.
$
One can check that these two functors may not form an extension of categories, since the condition (2) of Definition \ref{extension} may fail in general. In order to fix this problem, we need the following definition.

\begin{definition} \label{rgrocon}
 Let $\calC$ be a small category, $\mathfrak{A} \colon \calC \to k\mbox{\rm -Alg}$ be a precosheaf of unital $k$-algebras on $\calC$, and $\mathfrak{N} \colon \calC \to k\mbox{\rm -Mod}$ be a right $\mathfrak{A}$-module. Then a category $Gr_{\mathcal{C}}(\mathfrak{A}, \mathfrak{N})$, called \emph{the Grothendieck construction of $\mathfrak{A}$ and $\mathfrak{N}$}, is defined as follows:
\begin{enumerate}
    \item $\Ob Gr_{\mathcal{C}}(\mathfrak{A}, \mathfrak{N}):= \Ob \calC$;
    \item for each $x, y \in \Ob Gr_{\mathcal{C}}(\mathfrak{A}, \mathfrak{N})$,
    $$
    \Hom_{Gr_{\mathcal{C}}(\mathfrak{A}, \mathfrak{N})}(x, y) := \left\{ (r,m,f) ~|~ f \in \Hom_{\mathcal{C}}(x,y), r \in \mathfrak{A}(y), m \in \mathfrak{N}(y) \right\}
    $$
    \item for each $x, y, z \in \Ob Gr_{\mathcal{C}}(\mathfrak{A}, \mathfrak{N})$, and each $(r,m,f) \in \Hom_{Gr_{\mathcal{C}}(\mathfrak{A}, \mathfrak{N})}(x, y)$, $(s,n,g) \in \Hom_{Gr_{\mathcal{C}}(\mathfrak{A}, \mathfrak{N})}(y, z)$, we set
    $$
    (r,m,f) \circ (s,n,g) : = \left( \mathfrak{A}(g)(r)s, n + \mathfrak{N}(g)(m) \cdot s, fg \right).
    $$
\end{enumerate}
\end{definition}

\begin{remark}
The main difference between Definition \ref{rgrocon} and Definition \ref{bigrocon} is the composition law (3). In item (3) of Definition \ref{rgrocon} above, the left action is not allowed since $\mathfrak{N}$ is only a \emph{right} $\mathfrak{A}$-module.  
\end{remark}

When replacing $\mathfrak{M}$ by $\mathfrak{N}$, then an extension of categories will be obtained.

\begin{lemma} \label{exofcat}
Let $\mathfrak{N}_{\sqcup}:=\bigsqcup_{x \in \Ob \calC} \mathfrak{N}(x)$,
Then the following sequence of functors
$$
\mathfrak{N}_{\sqcup} \overset{\iota}{\longrightarrow} Gr_{\mathcal{C}}(\mathfrak{A}, \mathfrak{N}) \overset{\pi}{\longrightarrow} Gr_{\mathcal{C}}(\mathfrak{A})
$$ 
is an extension of categories, where the functor 
$
\iota \colon \mathfrak{N}_{\sqcup} \longrightarrow Gr_{\mathcal{C}}(\mathfrak{A}, \mathfrak{N})
$
is defined as:
$$
x \overset{m}{\longrightarrow} x \mapsto  x \overset{(1_{\mathfrak{A}(x)}, m, 1_x)}{\longrightarrow} x,
$$
and the functor
$
\pi \colon Gr_{\mathcal{C}}(\mathfrak{A}, \mathfrak{N}) \longrightarrow Gr_{\mathcal{C}}(\mathfrak{A})
$
is defined as:
$$
x \overset{(r,m,f)}{\longrightarrow} y \mapsto x \overset{(r,f)}{\longrightarrow} y.
$$
\end{lemma}

\begin{proof}
It can be proved by directly checking the conditions (1) and (2) of Definition \ref{extension}. By the definitions of $\iota$ and $\pi$, one can easily see that condition (1) is satisfied. It remains to check the condition (2). 

For two morphisms $(r,m_1,f), (r,m_2,f) \in \Hom_{ Gr_{\mathcal{C}}(\mathfrak{A}, \mathfrak{N})}(x, y)$, by computation, it is true that, $\pi \left((r,m_1,f)\right) = (r, f) = \pi \left((r,m_2,f)\right)$ if and only if there is a unique $m_3:= (m_2-m_1) \in \mathfrak{N}(y) $ such that 
\begin{align*}
     & (r,m_1,f) \circ \iota(m_3) & \\ 
    = & (r,m_1,f) \circ (1_{\mathfrak{A}(y)},m_3,1_y) & (Definition~\iota) \\
    = & (r,m_1,f) \circ (1_{\mathfrak{A}(y)},m_2-m_1,1_y) &  \\
    = & (\mathfrak{A}(1_y)(r)1_{\mathfrak{A}(y)}, (m_2-m_1)+\mathfrak{N}(1_y)(m_1) \cdot 1_{\mathfrak{A}(y)}, f1_y) & (Definition~\ref{rgrocon}~(3)) \\
    = & (r,m_2,f). & \\
\end{align*}
The uniqueness of $m_3$ is follows from the fact that the inverse elements are unique in an abelian group.
\end{proof}

\subsection{Spectral sequences}
In this subsection, two LHS-spectral sequences will be established. For more information about spectral sequences, one can see \cite{Xu08,Yal24} and McCleary's book \cite{Mcc01}. In order to obtain the spectral sequences, we need the following definition.

\begin{definition} \label{hfun}
The assignment $x \to \mathbb{H}^*(\mathfrak{N}(x); \mathfrak{F}(x))$ together with the induced homomorphisms $\mathbb{H}^*(\mathfrak{N}(y); \mathfrak{F}(y)) \to \mathbb{H}^*(\mathfrak{N}(x); \mathfrak{F}(x))$  defines \emph{a right $Gr_{\mathcal{C}}(\mathfrak{A})$-module}. We denote it by $\mathbb{H}^*(\mathfrak{N}_{\sqcup}; \mathfrak{F})$.    
\end{definition}

\begin{remark}
The existence of the induced morphisms in the definition above is referred to \cite[Lemma 7.7]{Yal24}.    
\end{remark}

With the extension of categories in Lemma \ref{exofcat}, we get our first LHS-spectral sequence.

\begin{theorem} \label{lhs1}
Let $\mathfrak{F} \in \mbox{\rm Mod-} Gr_{\mathcal{C}}(\mathfrak{A}, \mathfrak{N})$ and $\mathfrak{G} \in \mbox{\rm Mod-} Gr_{\mathcal{C}}(\mathfrak{A})$, then there is a spectral sequence
$$
\E_2^{p,q}=\Ext^p_{Gr_{\mathcal{C}}(\mathfrak{A})}(\mathfrak{G}, \mathbb{H}^q(\mathfrak{N}_{\sqcup}; \mathfrak{F})) \Longrightarrow \Ext^{p+q}_{Gr_{\mathcal{C}}(\mathfrak{A}, \mathfrak{N})}(\Res_{\pi} \mathfrak{G}, \mathfrak{F}).
$$
\end{theorem}

\begin{proof}
Apply \cite[Theorem 7.10]{Yal24} to the extension of categories in Lemma \ref{exofcat}.
\end{proof}

As an application of the above theorem, we obtain our second LHS-spectral sequence.

\begin{theorem} \label{lhs2}
Let $\mathfrak{F} \in \mbox{\rm Mod-} Gr_{\mathcal{C}}(\mathfrak{A}, \mathfrak{N})$, then there is a spectral sequence 
$$
\E_2^{p,q}=\H^p(Gr_{\mathcal{C}}(\mathfrak{A}); \mathbb{H}^q(\mathfrak{N}_{\sqcup}; \mathfrak{F})) \Longrightarrow \H^{p+q}(Gr_{\mathcal{C}}(\mathfrak{A}, \mathfrak{N}); \mathfrak{F}).
$$    
\end{theorem}

\begin{proof}
Let $\mathfrak{G}:=\underline{k}$, then it follows from Theorem \ref{lhs1}.    
\end{proof}


\bibliographystyle{plain}
\bibliography{ref}
\end{document}